\documentclass[12pt]{amsart}
\usepackage{amssymb}
\usepackage{amsmath,amssymb,amsfonts,amsthm,graphics,
latexsym, amscd, amsfonts, epsfig, eepic,epic,psfrag}
\usepackage{mathrsfs}
\usepackage{color}
\usepackage{graphicx}
\usepackage[all]{xypic}

\def\cE{\mathcal{E}}

\def\cT{\mathcal{T}}
\def\cA{\mathcal{A}}
\def\cB{\mathcal{B}}

\def\cO{\mathcal{O}}
\def\cS{\mathcal{S}}
\def\cR{\mathcal{R}}
\def\Cat{\mathsf{Cat}}

\theoremstyle{plain}

\newtheorem{theorem}{Theorem}
\newtheorem{lemma}{Lemma}
\newtheorem{proposition}{Proposition}
\newtheorem{corollary}{Corollary}
\newtheorem{conjecture}{Conjecture}

\theoremstyle{definition}
\newtheorem{definition}{Definition}

\theoremstyle{remark}

\numberwithin{equation}{section}

\begin{document}

\begin{center}
{\bf\Large  A combinatorial interpretation of the
            $\kappa^{\star}_{g}(n)$ coefficients
            }
\\
\vspace{15pt} Thomas J. X. Li, Christian M. Reidys
\end{center}

\begin{center}
Institut for Matematik og Datalogi,\\
University of Southern Denmark,\\
 Campusvej 55, DK-5230 Odense M, Denmark\\
Email: thomasli@imada.sdu.dk,  duck@santafe.edu
\end{center}
\centerline{\bf Abstract}{\small
Studying the virtual Euler characteristic of the moduli space of curves, Harer and Zagier
compute the generating function $C_g(z)$ of unicellular maps of genus
$g$. They furthermore identify coefficients, $\kappa^{\star}_{g}(n)$,
which fully determine the series $C_g(z)$.
The main result of this paper is a combinatorial interpretation of
$\kappa^{\star}_{g}(n)$. We show that these enumerate a class of
unicellular maps, which correspond $1$-to-$2^{2g}$ to a specific type
of trees, referred to as O-trees. O-trees are a variant of the
C-decorated trees introduced by Chapuy, F\'{e}ray and Fusy.
We exhaustively enumerate the number $s_{g}(n)$ of shapes of genus $g$
with $n$ edges, which is  a specific class of
unicellular maps with vertex degree at least three.
Furthermore we give combinatorial proofs for expressing the generating functions
$C_g(z)$ and $S_g(z)$
for unicellular maps and shapes in terms of $\kappa^{\star}_{g}(n)$, respectively.
We then prove a two term
recursion for $\kappa^{\star}_{g}(n)$ and that for any fixed $g$, the
sequence $\{\kappa_{g,t}\}_{t=0}^g$ is log-concave, where
$\kappa^{\star}_{g}(n)= \kappa_{g,t}$, for $n=2g+t-1$.




\section{Introduction}


A unicellular map is a connected graph embedded in a compact orientable surface,
in such a way that its complement is homeomorphic to a polygon. Equivalently,
a unicellular map of genus $g$ with $n$ edges can also be seen as gluing the edges of
$2n$-gon into pairs
to create an orientable surface of genus $g$.
It is  related to the general theory of  map enumeration,
the study of moduli spaces of curves \cite{Harer:86},
 the character theory of symmetric group \cite{Zagier:95,Jackson:87},
the computation of  matrix integrals \cite{lando:2004},
and also considered in a variety of application contexts \cite{penner:2010,reidys:2013}.
The most well-known example of unicellular maps is arguably
the class of plane trees, enumerated by the Catalan numbers (see for example \cite{stanley:2001}).

In \cite{Harer:86} Harer and Zagier study the virtual Euler characteristic of the moduli space of curves.
The number $\epsilon_g(n)$ counting the ways of gluing the edges of $2n$-gon
in order to obtain an orientable surface of genus $g$, i.e.~, the number of
unicellular maps of genus $g$ with $n$ edges turns out to play a crucial role
in their computations and they discover the two term recursion
\begin{equation}\label{E:bb}
(n+1)\epsilon_g(n)=2(2n-1)\epsilon_g(n-1)+ (n-1) (2n-1) (2n-3)\epsilon_{g-1}(n-2).
\end{equation}
Subsequently, they identify certain coefficients, $\kappa^{\star}_{g}(n)$,
which they describe to ``give the best coding of the information
contained in the [\ldots] series'', $\epsilon_g(n)$,~\cite{Harer:86}.
The key link is the following functional relation between the generating
function $K^{\star}_{g}(z)$ of $\kappa^{\star}_{g}(n)$ and the generating function
$C_g(z)$ of $\epsilon_g(n)$:
\[
C_g(z) = \frac{1}{\sqrt{1-4z}} \ K^{\star}_{g}\Big( \frac{z }{1-4z}\Big),
\]
where $K^{\star}_{g}(z)= \sum_{n=2g}^{3g-1}\kappa^{\star}_{g}(n) z^n$.

The main result of this paper is a combinatorial interpretation of the
$\kappa^{\star}_{g}(n)$ coefficients discovered by Harer and Zagier.
$\kappa^{\star}_{g}(n)$ enumerates a class unicellular maps, which
correspond $1$-to-$2^{2g}$ to certain O-trees.
O-trees are a variant of the C-decorated trees
introduced in~\cite{Chapuy:13}, see Theorem~\ref{T:ck}.
Using O-trees,
we exhaustively enumerate a specific class of
unicellular maps with vertex degree at least three, called shapes.
We give combinatorial proofs for expressing the generating functions $C_g(z)$ and $S_g(z)$
for unicellular maps and shapes in terms of $\kappa^{\star}_{g}(n)$, respectively.}
In particular the $\kappa^{\star}_{g}(n)$ are positive integers that
satisfy an analogue of eq.~(\ref{E:bb})
\begin{equation}\label{E:cc}
(n+1)\kappa^{\star}_{g}(n) =(n-1)(2n-1)(2n-3)
\kappa^{\star}_{g-1}(n-2)+ 2 (2n-1) (2n-3)(2n-5)\kappa^{\star}_{g-1}(n-3),
\end{equation}
see Corollary~\ref{C:k} and Theorem~\ref{T:kappa}. Eq.~(\ref{E:cc}) has been
independently discovered by Chekhov {\it et al.} \cite{Chekhov:14} using the matrix model.
We furthermore prove in Proposition~\ref{P:logk} that for any fixed $g$, the
sequence $\{\kappa_{g,t}\}_{t=0}^g$ is log-concave, where
$\kappa^{\star}_{g}(n)= \kappa_{g,t}$, for $n=2g+t-1$.
We conjecture that the
sequences $\{\kappa_{g,t}\}_{t=0}^g$ and $\{s_{g}(n)\}_{n=2g}^{6g-2}$  are infinitely log-concave.

\begin{figure}
\begin{center}
\includegraphics[width=0.7\textwidth]{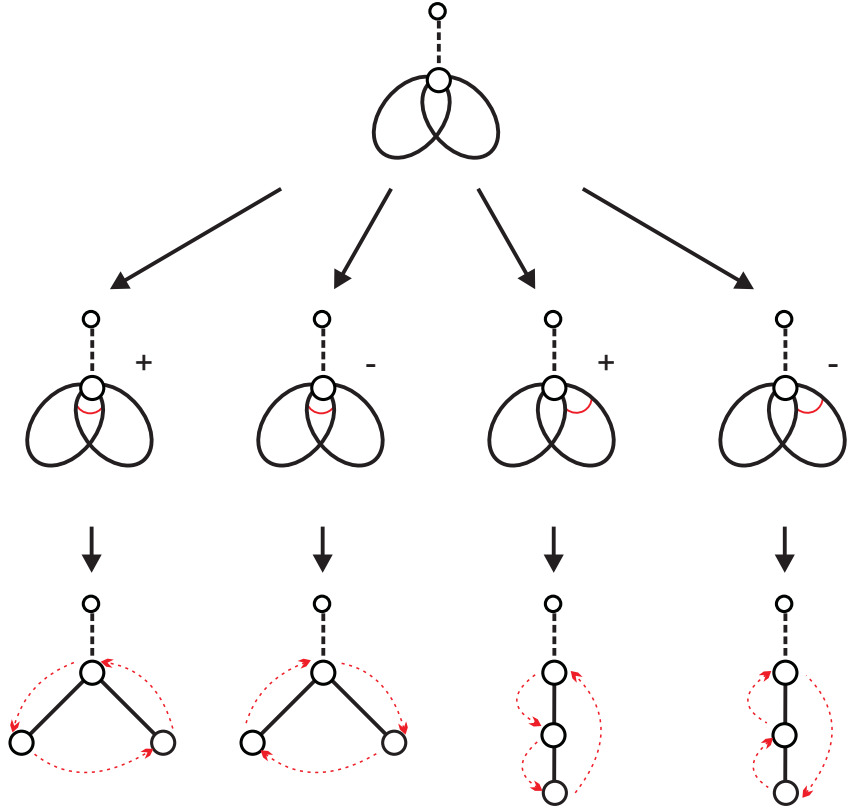}
\end{center}
\caption{\small $\kappa_{1,1}$ enumerates the unique unicellular maps of genus
$1$ with $2$ edges, which correspond $1$-to-$2^{2}$ to O-trees from $\cR_{1,1,0}$.
}\label{F:kappa}
\end{figure}

\section{Background}





A \emph{map} $M$ of genus $g\geq 0$ is a connected graph $G$ embedded on a
closed compact orientable surface $S$ of genus $g$, such that
$S\backslash G$ is homeomorphic to a collection of polygons, which are called the
\emph{faces}
of $M$. Loops and multiple edges are allowed. The (multi)graph $G$ is called the
\emph{underlying graph} of $M$ and $S$ its \emph{underlying surface}. 
Maps are considered up to homeomorphisms between the underlying surfaces.
A \emph{sector} of $M$ consists of two consecutive
edges around a vertex. A \emph{rooted map} is a map with a marked sector, called
the \emph{root};
the vertex incident to the root is called the \emph{root-vertex}.
In figures, we represent the root, by drawing a
dashed edge attaching the root-vertex and a distinguished vertex, called
the \emph{plant}.
By convention, the plant, plant-edge and its associated sector
(around the plant) are not considered, when counting the number of
vertices, edges or sectors.
From now on, all maps are assumed to be rooted and accordingly the underlying
graph of a rooted map is naturally vertex-rooted.
A \emph{unicellular map} is a map with a unique face.
By Euler's characteristic formula $|V|-|E|+|F|=2-2g$,
a unicellular map  of genus $g$ with $n$ edges has $n+1-2g$ vertices.
A \emph{plane tree} is a unicellular map of genus $0$.

We next introduce O-trees, which are directly implied by the concept of
C-decorated trees by Chapuy \textit{et al.}~\cite{Chapuy:13}.

An \emph{O-permutation} is a permutation where all cycles have odd length.
For each O-permutation $\sigma$ on $n$ elements,
the \emph{genus} of $\sigma$ is defined as $(n-\ell(\sigma))/2$,
where $\ell(\sigma)$ is the number of cycles of $\sigma$.

An \emph{O-tree} with $n$ edges is a pair $\alpha=(T,\sigma)$, where $T$
is a plane tree with $n$ edges and $\sigma$ is an O-permutation on $n+1$
elements, see Figure~\ref{F:ex}(a).
The \emph{genus} of $\alpha$ is defined to be the genus of $\sigma$.
We  canonically number the $n+1$ vertices of $T$  from $1$ to $n+1$
according to a left-to-right, depth-first traversal.
Hence $\sigma$ can be seen as a permutation of the vertices of $T$, see
Figure~\ref{F:ex}(b).

The \emph{underlying graph} $G(\alpha)$ of $\alpha$ is the (vertex-rooted)
graph $G$ with $n$ edges, that is obtained from $T$ by merging the vertices
in each cycle of $\sigma$ (so that the vertices of $G$ correspond to the
cycles of $\sigma$) into a single vertex, see Figure~\ref{F:ex}(c).
\begin{figure}[ht]
\begin{center}
\includegraphics[width=0.9\textwidth]{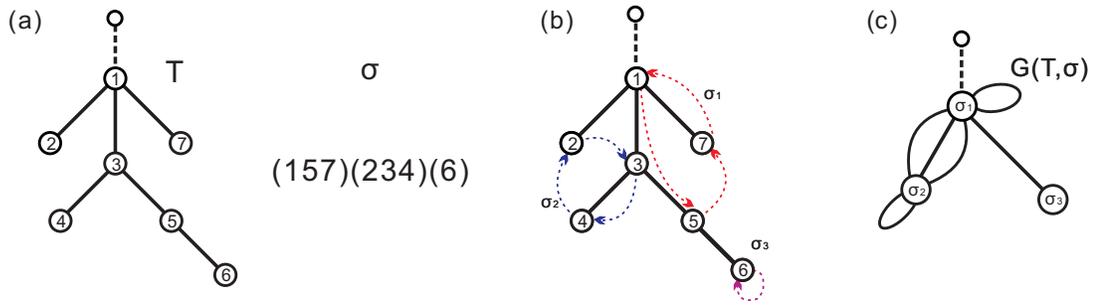}
\end{center}
\caption{\small An O-tree and its underlying graph.
}\label{F:ex}
\end{figure}

\begin{definition}
For $n,g$ nonnegative integers, let $\cE_g(n)$ denote the set of unicellular
maps of genus $g$ with $n$ edges, let  $\cO_g(n)$ be the set
of O-permutations of genus $g$ on $n$ elements  and let $\cT_g(n)$ denote the set
of O-trees of genus $g$ with $n$ edges, i.e., $\cT_g(n)=\cE_0(n)\times\cO_g(n+1)$.
\end{definition}




For two finite sets $\cA$ and $\cB$, let $\cA \uplus \cB$ denote their disjoint
union and $k\cA$  denote the set made of $k$ disjoint copies of $\cA$.
We write $\cA\simeq \cB$ if there exists a bijection between $\cA$ and
$\cB$.

Let us first recall a combinatorial result of \cite{Chapuy:11}:
\begin{proposition}[Chapuy~\cite{Chapuy:11}]\label{prop:1}
For $k\geq 1$, let $\cE_g^{(2k+1)}(n)$ denote the set of maps from $\cE_g(n)$
in which a set of $2k+1$ vertices is distinguished. Then for $g>0$ and
$n\geq 0$,
\begin{align}\label{eq:trisectionrec-maps}
2g\ \!\cE_g(n)\simeq
\biguplus_{k=1}^g \cE_{g-k}^{(2k+1)}(n).
\end{align}
In addition, if $M$ and $(M',S')$ are in correspondence, then the underlying
graph of $M$ is obtained from the underlying graph of $M'$ by merging the
vertices in $S'$ into a single vertex.
\end{proposition}

{\bf Remark:} one key feature of this bijection is that it preserves the
underlying graph of corresponding objects. By multiplying with the factor
$2^{2g}$ (which still preserves the underlying graph),  we obtain
\begin{equation}\label{eq:bi1}
2g\ \! 2^{2g} \cE_g(n)\simeq
\biguplus_{k=1}^g 2^{2k}\cdot 2^{2(g-k)}  \cE_{g-k}^{(2k+1)}(n).
\end{equation}

In analogy to the above decomposition of unicellular maps, there exists
a recursive way to decompose O-permutations:

\begin{proposition}[Chapuy \textit{et al.}~\cite{Chapuy:13}]\label{prop:2}
For $k\geq 1$, let $\cO_g^{(2k+1)}(n)$ be the set of O-permutations from $\cO_g(n)$
having $2k+1$ labeled cycles. Then for $g>0$ and $n\geq 0$,
\begin{equation}\label{eq:bi2}
2g\ \!\cO_g(n)\simeq
\biguplus_{k=1}^g 2^{2k}\cdot \cO_{g-k}^{(2k+1)}(n).
\end{equation}
Furthermore, if $\pi$ and $(\pi',S')$ are in correspondence,
then the cycles of $\pi$ are obtained from the cycles of $\pi'$
by merging labeled cycles in $S'$ into a single cycle.
\end{proposition}

Along these lines we furthermore observe:

\begin{proposition}[Chapuy \textit{et al.}~\cite{Chapuy:13}]\label{prop:3}
For $k\geq 1$, denote by $\cT_g^{(2k+1)}(n)$ the set of O-trees from $\cT_g(n)$
in which a set of $2k+1$ cycles is distinguished. Then for $g>0$ and $n\geq 0$,
$$
2g\ \!\cT_g(n)\simeq   \biguplus_{k=1}^g 2^{2k}\cdot \cT_{g-k}^{(2k+1)}(n).
$$
Furthermore, if $\alpha$ and $(\alpha',S')$ are in correspondence,
then the underlying graph of $\alpha$ is obtained from the underlying graph
of $\alpha'$ by merging the vertices corresponding to cycles from $S'$ into
a single vertex.
\end{proposition}

The proofs of Proposition~\ref{prop:2} and
Proposition~\ref{prop:3} are reformulations
 of those of \cite{Chapuy:13}
in the context of C-permutations and C-decorated trees. For completeness we give
them in the Appendix.

{\bf Remark:} the bijection for O-permutations preserves the cycles, which
implies that the bijection for O-trees preserves the underlying graph of
corresponding objects.



Combining Proposition~\ref{prop:1} and
Proposition~\ref{prop:3}, we inductively derive a bijection
preserving the underlying graphs.
\begin{theorem}[Chapuy \textit{et al.}~\cite{Chapuy:13}]\label{T:main}
For any non-negative integers $n$ and $g$, there exists a bijection
$$
2^{2g}\cE_g(n)\simeq\cT_g(n)= \cE_0(n)\times\cO_g(n+1).
$$
In addition, the cycles of an O-tree naturally correspond to the vertices
of the associated unicellular map, such that the respective underlying
graphs are the same.
\end{theorem}

{\bf Remark:} in~\cite{Chapuy:13}, Chapuy \textit{et al.} prove the existence of a
$1$-to-$2^{n+1}$ correspondence between C-decorated trees and unicellular maps.
The notion of C-permutation and C-decorated tree therein can be viewed as
O-permutation and O-tree carrying a sign with each cycle, respectively.
The reduction from C-decorated trees to O-trees allows us derive a
$1$-to-$2^{2g}$ correspondence between O-trees and unicellular maps.
Furthermore all the results in ~\cite{Chapuy:13} for C-decorated trees have an
O-tree analogue.

The proof of Theorem~\ref{T:main} is a reformulation of that for C-decorated trees
\cite{Chapuy:13}. We give its proof in the Appendix.


\section{Shapes}




\begin{definition}
A shape is a unicellular map having vertices of degree $\geq 3$.
\end{definition}

We adopt the convention that the plant-edge is taken into account
when considering the degree of the root vertex.

\begin{proposition}\cite{Huang:11}
Given a shape of genus $g$ with $n$ edges, we have $2g\leq n \leq 6g-2$.
\end{proposition}
\begin{proof}
By Euler's characteristic formula, we have $|V|=n+1-2g$,
where $V$ denotes the vertex set of a shape of genus $g$ with $n$ edges.
On the one hand, any shape contains at least one vertex, which implies
$|V|=n+1-2g \geq 1$, i.e., $n \geq 2g$. On the other hand, each vertex $v$
of a shape has $deg(v)\geq 3$. Then we derive $2(n+1)=\sum_{v\in V}deg(v)+1 \geq 3
|V| +1 =3 (n+1-2g)+1$, that is, $n\leq 6g-2$. (Here we consider the plant and
the plant-edge.)
\end{proof}

Let $\cS_g(n)$ denote the set of $\cE_g(n)$-shapes, i.e.~, shapes of genus $g$
with $n$ edges. Let $\cR_g(n)$ denote the set of O-trees from $\cT_g(n)$ such
that each vertex in the underlying graph of the O-tree
contains only vertices of degree $\geq 3$, that is
 \[
 \cR_g(n)=\{(T,\sigma)\in\cE_0(n)\times\cO_g(n+1)|
  \text{ each vertex of }G(T,\sigma) \text{ has degree }\geq 3 \}.
 \]

\begin{lemma}\label{L:shape}
For $g\geq 1 $ and $2g\leq n \leq 6g-2$, we have the bijection
 \[
 2^{2g}\cS_g(n)\simeq\cR_g(n).
\]
In addition, the cycles of an O-tree naturally correspond to the vertices
of the associated unicellular map, in such a way that the respective underlying
graphs are the same.
\end{lemma}

Note that a unicellular map is a shape if and only if each vertex in the
underlying graph of the map has degree $\geq 3$. Therefore, Lemma~\ref{L:shape}
follows directly from Theorem~\ref{T:main} by restricting the bijection to the
set $\cS_g(n)$ of shapes since the bijection therein
preserves the underlying graph of corresponding objects.

Lemma~\ref{L:shape} allows us to obtain deeper insight into shapes via
O-permutations. To this end we consider the cycle-type of an O-permutation,
i.e., a partition with parts of odd size.
Given an O-permutation from $\cO_g(n+1)$, its cycle-type is a partition $\beta$
of $n+1$ with $n+1-2g$ odd parts. We assume that
$\beta=1^{n+1-2g-t}3^{m_1}\ldots (2j+1)^{m_j}$, with  $t=m_1+\cdots+m_j$.
The partition $\beta=1^{n+1-2g-t}3^{m_1}\ldots (2j+1)^{m_j}$ naturally corresponds to
the partition $\gamma=1^{m_1}\ldots j^{m_j}$ of $g$.
The fact that $\gamma$ is a partition of $g$ follows from the identity
$(n+1-2g-t)+ 3 m_1+\cdots (2j+1) m_j =n+1$. Here $t=m_1+\cdots+m_j=\ell(\gamma)$
denotes the number of odd parts $>1$ of $\beta$, i.e., the number of parts of
$\gamma$. Let $k=n+1-2g-t$ denote number of parts $=1$ of  $\beta$.
It is clear that this a one-to-one correspondence.
Therefore the cycle type $\beta$ of an O-permutation from $\cO_g(n+1)$ can be
indexed by an partition $\gamma$ of $g$.

The number $a_{\gamma}(k)$ of O-permutations of \hbox{$n+1=2g+t+k$}
elements with cycle-type equal to $\beta=1^{k}3^{m_1}\ldots (2j+1)^{m_j}$
is given by
$$
a_{\gamma}(k)=\frac{(2g+t+k)!}{k!\prod_im_i!(2i+1)^{m_i}},
$$
where $\gamma=1^{m_1}\ldots j^{m_j}$.

Let $\cO_{g,t,k}$ denote the set of O-permutations of genus $g$ with $k$ cycles of
length $1$ and $t$ cycles of length $>1$.
Note that the number of elements of an O-permutation from $\cO_{g,t,k}$ is
$n+1= 2g+t+k$. Then we have the following two cases:
\begin{enumerate}
\item For $k=0$, the cardinality $\cO_{g,t,0}$, denoted by $a_{g,t}$, counts
O-permutations of genus $g$ on \hbox{$2g+t$} elements without cycles of length $1$
(or cycle-type having the form $\beta=3^{m_1}\ldots (2j+1)^{m_j}$).
Hence it is  given by
$$
a_{g,t}=\sum_{\substack{\gamma \vdash g \\ \ell(\gamma)=t} } a_{\gamma}(0) =
(2g+t)!\sum_{\substack{\gamma \vdash g \\ \ell(\gamma)=t} }  \frac{1}{\prod_im_i!(2i+1)^{m_i}},
$$
where $\gamma=1^{m_1}2^{m_2}\cdots j^{m_j}$ runs over all partitions of $g$ with
$t$ parts.

\item For arbitrary $k$, each O-permutation in $\cO_{g,t,k}$
consists of an O-permutation from $\cO_{g,t,0}$ together with
$k$ cycles of length $1$. Then
the set $\cO_{g,t,k}$ can be counted by first picking up $k$ elements
from $2g+t+k$ elements and then choosing an O-permutation from $\cO_{g,t,0}$.
Therefore
$$
|\cO_{g,t,k}|={2g+t+k \choose k} a_{g,t}=
\frac{(2g+t+k)!}{k!}\sum_{\substack{\gamma \vdash g \\ \ell(\gamma)=t} }
 \frac{1}{\prod_im_i!(2i+1)^{m_i}}.
$$
\end{enumerate}

By definition,
\[
 \cO_g(n+1)= \biguplus_{t+k=n+1-2g}\cO_{g,t,k}.
\]


Set $n+1=2g+t+k$.
 Let $\cR_{g,t,k}$ denote
 the set of O-trees from $\cR_g(n)$ such that their associated
O-permutation has $k$ cycles of length $1$ and $t$ cycles of length $>1$, i.e.,
 \[
 \cR_{g,t,k}=\{(T,\sigma)\in\cE_0(n)\times\cO_{g,t,k}| \text{ each vertex of }
  G(T,\sigma) \text{ has degree }\geq 3 \}.
 \]
Hence
\[
 \cR_g(n)= \biguplus_{t+k=n+1-2g}\cR_{g,t,k}.
\]

\begin{lemma}
For $k=0$, we have
\[
\cR_{g,t,0}= \cE_0(2g+t-1)\times\cO_{g,t,0}.
\]
Therefore
\[
|\cR_{g,t,0}|= \Cat(2g+t-1)\, a_{g,t}=\frac{(2(2g+t-1))!}{(2g+t-1)!}
\sum_{\substack{\gamma \vdash g \\ \ell(\gamma)=t} }  \frac{1}{\prod_im_i!(2i+1)^{m_i}},
\]
where $\Cat(n):=\frac{(2n)!}{n! (n+1)!}$ is the $n$-th Catalan number
and $\gamma=1^{m_1}2^{m_2}\cdots j^{m_j}$ runs over all partitions of $g$ with $t$ parts.
\end{lemma}
\begin{proof}
By definition, $\cR_{g,t,0}\subseteq \cE_0(2g+t-1)\times\cO_{g,t,0}$.
Given $(T,\sigma)\in \cE_0(2g+t-1)\times\cO_{g,t,0}$, each cycle of O-permutation
$\sigma$ has length $\geq 3$.
Then the underlying graph $G(T,\sigma)$ of $(T,\sigma)$, obtained
from $T$ by merging into a single vertex the vertices in each cycle of $\sigma$,
must have all vertices with degree $\geq 3$. It implies that $(T,\sigma)\in\cR_{g,t,0}$.
Hence $\cR_{g,t,0}= \cE_0(2g+t-1)\times\cO_{g,t,0}$.
\end{proof}

To enumerate O-trees from $\cR_{g,t,k}$ for arbitrary $k$, we observe that they can be
reduced to O-trees from $\cR_{g,t,0}$. The key idea is to eliminate the vertices
corresponding to $1$-cycles from an O-tree, thereby reducing to an O-tree without
$1$-cycles, i.e., O-tree from $\cR_{g,t,0}$.
This elimination on O-trees is reminiscent of R\'{e}my's bijection~\cite{Remy:85} on plane trees,
which is briefly recalled below.

R\'{e}my's bijection reduces a plane tree $T$ with $n$ edges
and a labeled vertex to a plane tree $T'$ with $n-1$ edges
and a sector labeled by $+$ or $-$ as follow
\begin{itemize}
\item if the labeled vertex is a leaf, $T'$ is obtained from  $T$ by
      contracting the edge connecting the labeled vertex and its father.
      Label by $+$ the sector associated with the labeled vertex,
\item if the labeled vertex is a non-leaf, $T'$ is obtained from  $T$ by
      contracting the edge connecting the labeled vertex and its leftmost
      child. Label by $-$ the sector separating the leftmost subtree
      and the remaining subtree of the labeled vertex.
 \end{itemize}
Therefore $(n+1)\cE_0(n) \simeq 2(2n-1)\cE_0(n-1)$, see Figure~\ref{F:remy}.
R\'{e}my's bijection has been applied in~\cite{Chapuy:13,Huang:14,Li}.

\begin{figure} \begin{center}
\includegraphics[width=0.6 \textwidth]{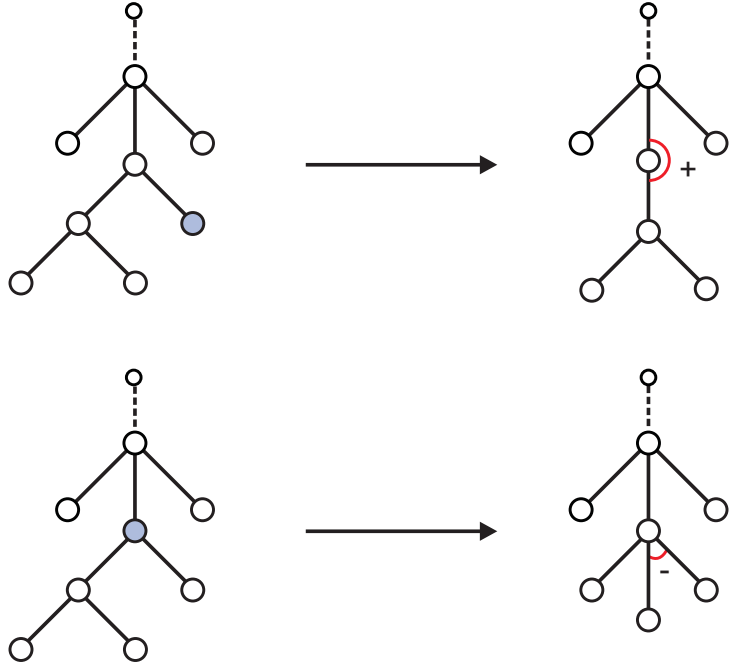}
\caption{R\'{e}my's bijection.}
\label{F:remy}
\end{center} \end{figure}

Given an O-tree $(T,\sigma)\in \cR_{g,t,k}$, its \emph{traversal}
is defined as that of its underlying plane tree
(traveling around the boundary of $T$ starting from the root-sector).
A vertex $v$ of $(T,\sigma)$ is called a \emph{$1$-cycle} if the corresponding
element in $\sigma$ is in a cycle of length $1$.
All sectors around the vertex $v$ are ordered according to the traversal
of $(T,\sigma)$. A sector $\tau$ at $v$ in $(T,\sigma)$ is called \emph{permissible}
if
\begin{itemize}
\item $\tau$ is not the last sector around $v$  according to the traversal of $T$,
\item if vertex $v$ is $1$-cycle, then $\tau$ is not the first two sectors around $v$
      according to the traversal of $T$.
\end{itemize}

\begin{proposition}\label{prop:persec}
Any O-tree $(T,\sigma) \in\cR_{g,t,k}$ has exactly $(2g+t-k-1)$ permissible sectors.
\end{proposition}
\begin{proof}
By definition of $\cR_{g,t,k}$, each vertex in the  underlying
graph $G(T,\sigma)$ has degree $\geq 3$. Then
each $1$-cycle vertex has degree $\geq 3$ in $T$
since it has the same degree as its corresponding vertex in the underlying graph.

Accordingly, any $1$-cycle vertex has no less than $3$ sectors, whence
its first two sectors never coincide with its last sector.
Note that any $(T,\sigma)\in \cR_{g,t,k}$ has $(2g+t+k-1)$ edges, $2(2g+t+k-1)+1$
sectors and $k$ $1$-cycle vertices. Thus in $(T,\sigma)$, the set of permissible sectors
is the set of all $(T,\sigma)$-sectors excluding all last sectors of vertices
and all the first two sectors of $1$-cycle vertices.
Hence the number of permissible $(T,\sigma)$-sectors is given by
$2(2g+t+k-1)+1-(2g+t+k)-2 k = 2g+t-k-1$.
\end{proof}

Let $\cR_{g,t,k}^{(l)}$ denote the set of $\cR_{g,t,k}$ O-trees with
$l$ permissible, labeled sectors. By Proposition~\ref{prop:persec},
$0\leq l \leq 2g+t-k-1$ and $|\cR_{g,t,k}^{(l)}|= {2g+t-k-1 \choose l}
|\cR_{g,t,k}|$.

Let $\cR_{g,t,k}^{v}$ denote the set of $\cR_{g,t,k}$ O-trees with
one  labeled $1$-cycle vertex.
Since each $\cR_{g,t,k}$ O-tree has $k$ $1$-cycle vertices,
we have $|\cR_{g,t,k}^{v}|= k |\cR_{g,t,k}|$.

\begin{lemma}\label{L:shape1}
For $k\geq 1$, there exists a bijection between $\cR_{g,t,k}^{v}$, the set
of $\cR_{g,t,k}$ O-trees with
one  labeled $1$-cycle vertex and $\cR_{g,t,k-1}^{(1)}$, the set
of $\cR_{g,t,k-1}$ O-trees with one permissible, labeled sector.
Accordingly we have
\[k |\cR_{g,t,k}|= (2g+t-k)  |\cR_{g,t,k-1}| .\]
\end{lemma}

\begin{proof}
Suppose we are given $(T,\sigma,v)\in \cR_{g,t,k}^{v}$, and a $1$-cycle vertex $v$.
$v$ has degree $\geq 3$ and is not a leaf.

We construct an O-tree $(T',\sigma',s)$ from $\cR_{g,t,k-1}^{(1)}$ as follows:
we apply R\'{e}my's bijection to the plane tree $T$ with respect to the non-leaf $v$,
i.e., contracting the edge connecting $v$ and its leftmost
child. We obtain a plane tree $T'$ together with a labeled sector $s$.
The correspondence between vertices in  $T$ and those in $T'$
gives us a canonical relabelling of elements of the permutation $\sigma$
excluding the $1$-cycle corresponding to $v$.
Let $\sigma'$ denote the permutation obtained from $\sigma$
by deleting $1$-cycle corresponding to $v$ and relabeling,
see Figure~\ref{F:sbij}.

We define the mapping
$$
\Pi \colon \cR_{g,t,k}^{v}  \to \cR_{g,t,k-1}^{(1)}, \quad
(T,\sigma,v) \mapsto (T',\sigma',s).
$$

\begin{figure}
\begin{center}
\includegraphics[width=0.9\textwidth]{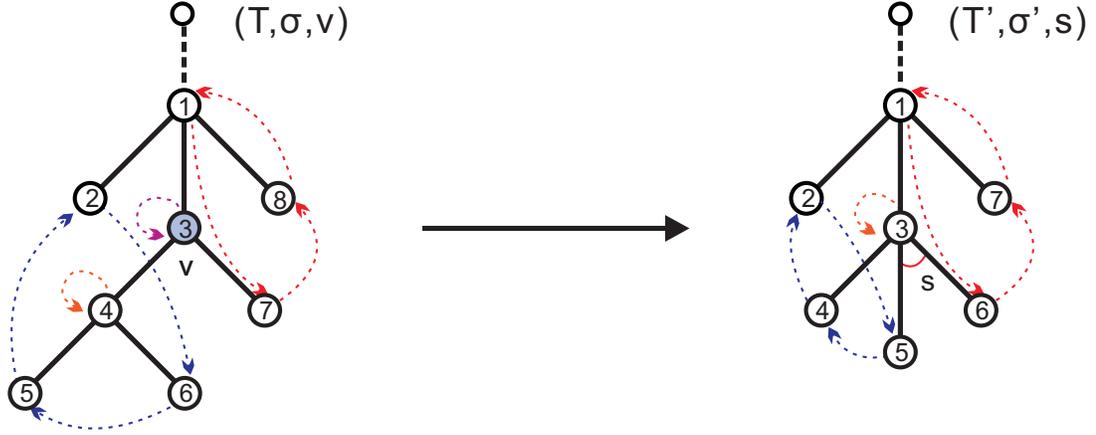}
\end{center}
\caption{\small The bijection $\Pi$.
}\label{F:sbij}
\end{figure}

First we show that $\Pi$ is well-defined.
By construction, $T'$ has $2g+t+k-2$ edges and
all cycles in $\sigma'$ have odd length, i.e., $\sigma'$ is an O-permutation.
Further $\sigma'$  has  $2g+t+k-1$ elements,
$k-1$ cycles of length $1$ and $t$ odd cycles of length $>1$.
Hence $(T',\sigma')\in \cR_{g,t,k-1}$.
Let $v'$ denote the vertex in $T'$, to which sector $s$ belongs to.
$s$ is not the last sector around $v'$, since otherwise, by construction of
R\'{e}my's bijection, the $1$-cycle vertex $v$ in $T$ has degree at most two,
a contradiction.
If $v'$ is a $1$-cycle in $(T',\sigma')$, then
the leftmost child of $v$ is a $1$-cycle in $(T,\sigma)$ and $v_1$ has degree
at least $\geq 3$ in $T$. By the way of contracting the edge connecting $v$ and $v_1$
and labeling the sector, $s$ is then not one of the first two sectors around $v'$.
This shows that $s$ is permissible, whence $(T',\sigma',s)\in \cR_{g,t,k-1}^{(1)}$ and
$\Pi$ is well-defined.

To recover $(T,\sigma,v)\in \cR_{g,t,k}^{v}$ from $(T',\sigma',s)\in \cR_{g,t,k-1}^{(1)}$,
we apply the inverse of R\'{e}my's bijection to $T'$ with respect to the sector $s$,
which is labeled $-$ and obtain a plane tree $T$ with non-leaf vertex $v$.
By construction of
R\'{e}my's bijection and the definition of permissible sector, $v$ has degree $\geq 3$.
Set $\sigma$ to be the permutation obtained from $\sigma'$ by adding the $1$-cycle
corresponding to $v$ and relabeling according to the correspondence between vertices
in  $T'$ and those in $T$, see Figure~\ref{F:sbij}.
It is clear that this is the inverse of $\Pi$, whence $\Pi$ is bijective.
\end{proof}

By applying Lemma~\ref{L:shape1} successively, we derive
\begin{eqnarray*}
|\cR_{g,t,k}|&=& \frac{2g+t-k}{k}  |\cR_{g,t,k-1}|\\
&=& \frac{2g+t-k}{k} \cdot \frac{2g+t-k+1}{k-1}  \cdot |\cR_{g,t,k-2}| \\
&=&\cdots \\
&=& {2g+t-1 \choose k } |\cR_{g,t,0}| .
\end{eqnarray*}
Accordingly, Lemma~\ref{L:shape1} induces a bijection from $\cR_{g,t,k}$ to
$\cR_{g,t,0}^{(k)}$:
\begin{lemma}\label{L:shape2}
For any $k$, there exists a bijection
from $\cR_{g,t,k}$ to $\cR_{g,t,0}^{(k)}$, the set of $\cR_{g,t,0}$ O-trees with
$k$ permissible, labeled  sectors:
\[|\cR_{g,t,k}|= {2g+t-1 \choose k } |\cR_{g,t,0}| .\]
\end{lemma}

{\bf Remark:} given an $\cR_{g,t,k}$ O-tree, the number $k$ of $1$-cycle vertices
              is bounded by $k\leq 2g+t-1$.

In Figure~\ref{F:shape}, we show how to generate all $\cR_{1,1,1}$ and $\cR_{1,1,2}$ O-trees
from $\cR_{1,1,0}$ O-trees.

\begin{figure}
\begin{center}
\includegraphics[width=0.9\textwidth]{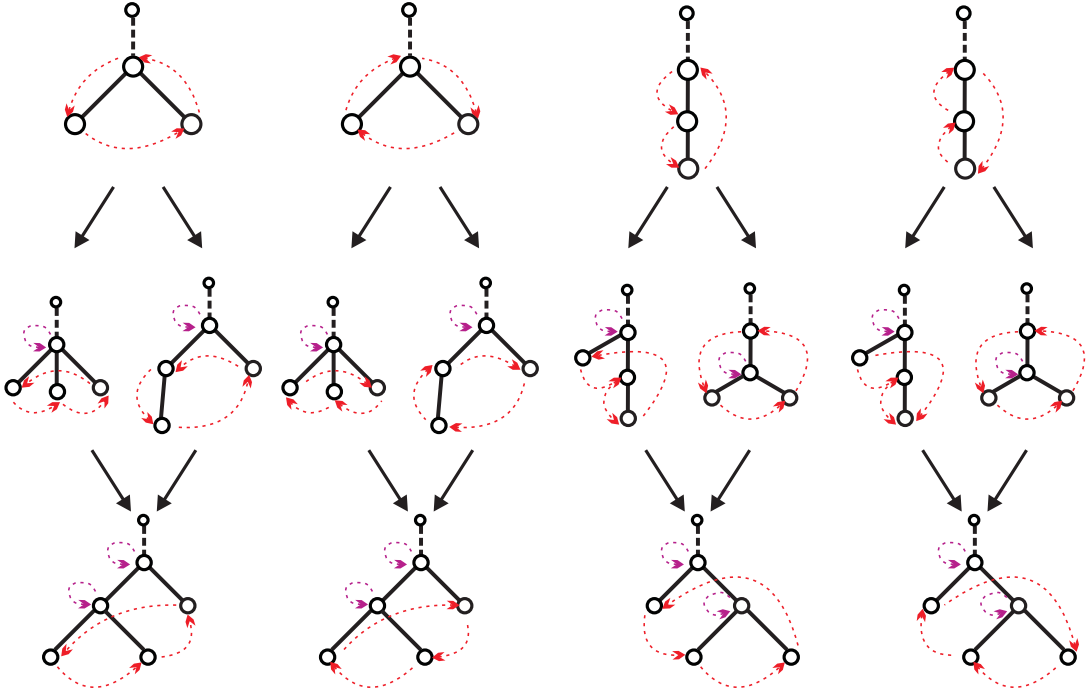}
\end{center}
\caption{\small Generation of all $\cR_{1,1,1}$ and $\cR_{1,1,2}$ O-trees
from $\cR_{1,1,0}$ O-trees.
}\label{F:shape}
\end{figure}

Let
\[
\kappa_{g,t}= \frac{|\cR_{g,t,0}|}{2^{2g}} =\frac{(2(2g+t-1))!}{2^{2g} (2g+t-1)!}
\sum_{\substack{\gamma \vdash g \\ \ell(\gamma)=t} }  \frac{1}{\prod_im_i!(2i+1)^{m_i}},
\]
where $\gamma=1^{m_1}2^{m_2}\cdots j^{m_j}$ is a partition of $g$ with $t$ parts.

For $g\geq 1$, let $s_g(n)$ be the number of shapes of genus $g$ with $n$ edges
and $S_g(z)$ denote the corresponding generating polynomial
$S_g(z)=\sum_{n=2g}^{6g-2} s_g(n) z^n$ . Then
\begin{lemma}\label{L:s}\cite{Huang:14}
For any $g\geq 1$, the generating polynomial of shapes is given by
\[
S_g(z)=\sum_{t=1}^{g} \kappa_{g,t} z^{2g+t-1} (1+z)^{2g+t-1}.
\]
\end{lemma}
\begin{proof}
By Lemma~\ref{L:shape}, we have
$2^{2g} \cS_g(n)\simeq \cR_g(n) = \biguplus_{t+k=n+1-2g} \cR_{g,t,k}$ and furthermore
 \[
2^{2g}\biguplus_{n=2g}^{6g-2} \cS_g(n)\simeq \biguplus_{n=2g}^{6g-2} \cR_g(n) =
\biguplus_{t=1}^{g} \biguplus_{k=0}^{2g+t-1} \cR_{g,t,k}.
\]
Therefore, by Lemma~\ref{L:shape2},
\begin{eqnarray*}
S_g(z) &=& \sum_{n=2g}^{6g-2} |\cS_g(n)| z^n\\
       &=& \frac{1}{2^{2g}} \sum_{t=1}^{g}  \sum_{k=0}^{2g+t-1}  | \cR_{g,t,k}| z^{2g+t+k-1}\\
       &=& \frac{1}{2^{2g}} \sum_{t=1}^{g}  \sum_{k=0}^{2g+t-1} {2g+t-1 \choose k } |\cR_{g,t,0}| z^{2g+t+k-1}\\
      &=& \frac{1}{2^{2g}} \sum_{t=1}^{g}  |\cR_{g,t,0}| z^{2g+t-1} \sum_{k=0}^{2g+t-1} {2g+t-1 \choose k } z^{k}\\
       &=& \frac{1}{2^{2g}} \sum_{t=1}^{g}  |\cR_{g,t,0}| z^{2g+t-1} (1+z)^{2g+t-1}\\
      & =& \sum_{t=1}^{g} \kappa_{g,t} z^{2g+t-1} (1+z)^{2g+t-1}.
\end{eqnarray*}
\end{proof}

\begin{corollary}
We have
\begin{equation} \label{E:s}
s_g(n)=\sum_{t=1}^{g} \kappa_{g,t}{2g+t-1 \choose n -(2g+t-1) },
\end{equation}
where $ {n \choose k} =0$ if $k<0$ or $k>n$.
\end{corollary}
\begin{proof}
By Lemma~\ref{L:s}, we have
\[
\sum_{n=2g}^{6g-2} s_g(n) z^n =\sum_{t=1}^{g} \kappa_{g,t} z^{2g+t-1} (1+z)^{2g+t-1}
=\sum_{t=1}^{g} \sum_{i=0}^{2g+t-1}\kappa_{g,t} {2g+t-1 \choose i}z^{2g+t-1+i}.
\]
Set $n=2g+t-1+i$.
By comparing both sides of the above identity, we obtain the corresponding formula for
$s_g(n)$.
\end{proof}
\begin{corollary}\label{C:k}
The number $\kappa_{g,t}$ is a positive integer.
\end{corollary}
\begin{proof}
The positivity of $\kappa_{g,t}$ is clear by definition.
We proceed by induction on $t$: assume that $\kappa_{g,j}$ is an integer for
$j< t$ and set $n=2g+t-1$. By eq.~(\ref{E:s}), we have
\[
s_g(2g+t-1)=\sum_{j=1}^{t} \kappa_{g,j}{2g+j-1 \choose t-j },
\]
i.e.,
\[
\kappa_{g,t} =s_g(2g+t-1) -\sum_{j=1}^{t-1} \kappa_{g,j}{2g+j-1 \choose t-j }.
\]
Since $s_g(2g+t-1)$ and $\kappa_{g,j}$ are integers for $j< t$, $\kappa_{g,t}$ is an integer.
\end{proof}

\section{The coefficients $\kappa^\star_g(n)$}\label{S:kappa*}

Let $\epsilon_g(n)$ denote the number of unicellular maps of genus $g$ with $n$ edges.
In the following we derive an explicit formula for the generating function of
unicellular maps of genus $g$, which has the same coefficients $\kappa_{g,t}$
as in the generating polynomial of shapes of genus $g$ in Lemma~\ref{L:s}.
This result has been observed in \cite{Huang:14} by a different construction.
\begin{lemma}\label{L:c}
For any $g\geq 1$, the generating function of unicellular maps of genus $g$ is given by
\[
C_g(z)=\sum_{t=1}^{g} \frac{\kappa_{g,t} z^{2g+t-1}} {(1-4z)^{2g+t-\frac{1}{2}}}.
\]
\end{lemma}
\begin{proof}
Note that
$\cO_g(n+1) =\biguplus_{t=1}^{g} \cO_{g,t,k}$ and
$|\cO_{g,t,k}|= {n+1 \choose 2g+t} a_{g,t}$, where $k=n+1-2g-t$.
Thus $|\cO_g(n+1)|= \sum_{t=1}^{g} {n+1 \choose 2g+t} a_{g,t}$.
By Theorem~\ref{T:main}, $2^{2g}\cE_g(n)\simeq\cE_0(n)\times \cO_g(n+1)$ and
we have
\begin{eqnarray*}
\epsilon_g(n) &=& \frac{1}{2^{2g}} \Cat(n)
\sum_{t=1}^{g} {n+1 \choose 2g+t} a_{g,t}\\
&=&
\sum_{t=1}^{g} \frac{(2n)!}{2^{2g} n!(n+1-2g-t)!(2g+t)!}   a_{g,t}.
\end{eqnarray*}

Therefore
using
\[
  \sum_{n\geq r-1}  \frac{(2n)!}{ n!(n+1-r)!} z^n =
             \frac{(2(r-1))! }{(r-1)!} \frac{ z^{r-1}} {(1-4z)^{r-\frac{1}{2}}},
\]
we compute
\begin{eqnarray*}
C_g(z) &=&  \sum_{n\geq 2g} \epsilon_g(n) z^n\\
       &=&  \sum_{n\geq 2g} \sum_{t=1}^{g} \frac{(2n)!}{2^{2g} n!(n+1-2g-t)!(2g+t)!}   a_{g,t} z^n\\
       &=& \sum_{t=1}^{g}  \frac{a_{g,t}}{2^{2g}(2g+t)!}  \sum_{n\geq 2g}\frac{(2n)!}{ n!(n+1-2g-t)!}  z^n\\
       &=& \sum_{t=1}^{g}  \frac{a_{g,t}}{2^{2g}(2g+t)!} \cdot \frac{(2(2g+t-1))!}{(2g+t-1)!}\cdot \frac{z^{2g+t-1}}{(1-4z)^{2g+t-\frac{1}{2}}}\\
     &=& \sum_{t=1}^{g} \frac{\Cat(2g+t-1) a_{g,t}}{2^{2g}}\cdot  \frac{ z^{2g+t-1}} {(1-4z)^{2g+t-\frac{1}{2}}} \\
    &=& \sum_{t=1}^{g} \frac{\kappa_{g,t} z^{2g+t-1}} {(1-4z)^{2g+t-\frac{1}{2}}}.
\end{eqnarray*}
\end{proof}

Let
$$
K^{\star}_{g}(z)=\sum_{n=2g}^{3g-1}\kappa^{\star}_{g}(n) z^n,
$$
then \cite{Harer:86} shows that
\[
C_g(z) = \frac{1}{\sqrt{1-4z}} \ K^{\star}_{g}\Big( \frac{z }{1-4z}\Big).
\]
In view of Lemma~\ref{L:shape} and Lemma~\ref{L:c} this provides the following
combinatorial interpretation of $\kappa_g^\star(n)$:

\begin{theorem}\label{T:ck}
$\kappa^{\star}_{g}(n)= \kappa_{g,t}$, where $n=2g+t-1$ and $\kappa_g^\star(n)$ counts
the shapes of genus $g$, which correspond to $\cR_{g,t,0}\subset \cR_g(n)$ via
the bijection in Lemma~\ref{L:shape}.
\end{theorem}

In \cite{reidys:2013}, $C_g(z)$ has been shown to have the form
\[
 C_g(z)= \frac{P_g(z)}{(1-4z)^{3g-\frac{1}{2}}},
\]
where $P_g(z)$ is a polynomial with integer coefficients.

Combining this with Lemma~\ref{L:c}, we obtain an explicit formula for the
polynomials $P_g(z)$ in terms of $\kappa_{g,t}$:
\begin{corollary}
For any $g\geq 1$, the polynomial $P_g(z)$ is given by
\[
P_g(z)=\sum_{t=1}^{g} \kappa_{g,t} z^{2g+t-1} (1-4z)^{g-t}.
\]
\end{corollary}
Combining Lemma~\ref{L:s} and Lemma~\ref{L:c}, we also derive the following functional
relation between $C_g(z)$ and $S_g(z)$
\begin{corollary}
For $g \geq 1$, we have
\begin{equation}
C_g(z)= \frac{1+z C_0(z)^2}{1-z C_0(z)^2}
S_g\left(\frac{z C_0(z)^2}{1-z C_0(z)^2}\right),
\end{equation}
where the generating function $C_0(z)$ of plane trees with $n$ edges
is given by $C_0(z)=\sum_n \epsilon_0(n) z^n = \frac{1-\sqrt{1-4z}}{2z}$
\end{corollary}
This functional relation can also be derived via symbolic methods \cite{Flajolet:07a}.
More precisely, we can construct a general unicellular map from a shape
by first replacing each edge by a path and then attaching a plane tree
to each sector.

We shall proceed by giving a bijective proof of a recurrence of $a_{g,t}$.

\begin{proposition}
For any $1\leq t\leq g$,
there exists a bijection
\[
\cO_{g,t,0} \simeq  (2g+t-1) (2g+t-2) (\cO_{g-1,t,0} + \cO_{g-1,t-1,0} ).
\]
Therefore $a_{g,t}$ satisfies the recurrence
\begin{equation}\label{E:arec}
a_{g,t}=(2g+t-1) (2g+t-2) (a_{g-1,t}+a_{g-1,t-1}),
\end{equation}
with $a_{1,1}=2$ and $a_{g,t}=0$ if $t<1$ or $t> g$. The values
$a_{g,t}$ for $g\leq 5$ are given in Table~\ref{Tab:a}.
\end{proposition}

\begin{table}
\begin{center}
\begin{tabular}{c|ccccc}
\small$a_{g,t}$ & \small$t=1$       & \small$2$         & \small$3$         & \small$4$         & \small $5$ \\
\hline
\small $g=1$    & \small$2$         & \small            & \small            & \small            & \small  \\
\small $2$      & \small$24$        & \small $40$       & \small            & \small            & \small  \\
\small $3$      & \small$720$       & \small $2688$     & \small$2240$      & \small            & \small  \\
\small $4$      & \small$40320$     & \small $245376$   & \small$443520$    & \small  $246400$  & \small  \\
\small $5$      & \small$3628800$   & \small $31426560$ & \small$90934272$  & \small  $107627520$ & \small $44844800$
\end{tabular}
\end{center}
\caption{\small $a_{g,t}$ of O-permutations of genus $g$
on \hbox{$2g+t$} elements having no cycles of
length $1$ and $t$ cycles of length $>1$.}
\label{Tab:a}
\end{table}

\begin{proof}
Set $n=2g+t$. Let $\mathcal{F}$ and $\mathcal{G}$ denote the subsets of $\cO_{g,t,0}$
where the cycle containing element $n$ has length $3$ and greater than $3$, respectively.
For any $\sigma \in \cO_{g,t,0}$, we have two scenarios
\begin{enumerate}
\item if $\sigma \in\mathcal{G}$, assume the cycle $c$ containing $n$ is
of the form $(h,\ldots,i,j,n)$. Removing $j$ and $n$ from $c$, we obtain an O-permutation
with $2g+t-2$ elements and $t$ cycles, which, after natural relabeling,
 corresponds to an O-permutation $\sigma'$ contained in $\cO_{g-1,t,0}$.
There are $2g+t-1$ ways to choose $j$ and $2g+t-2$ ways to insert $j$ and $n$.
Thus $\mathcal{G}$ is in bijection with $ (2g+t-1) (2g+t-2)\cO_{g-1,t,0}$,

\item if $\sigma \in\mathcal{F}$, then the cycle $c$ containing $n$ is
of the form $(i,j,n)$. By deleting $c$ from $\sigma$, we obtain an O-permutation
with $2g+t-3$ elements with $t-1$ cycles, which, after natural relabeling,
 corresponds to an O-permutation $\sigma'$ contained in $\cO_{g-1,t-1,0}$.
The number of ways to choose $i,j$ is $(2g+t-1) (2g+t-2)$, whence $\mathcal{F}$
is in bijection with $ (2g+t-1) (2g+t-2)\cO_{g-1,t-1,0} $.
\end{enumerate}
Since $\cO_{g,t,0}=\mathcal{G}\uplus\mathcal{F} $, we have a bijection
\[
\beta\colon\cO_{g,t,0} \to (2g+t-1) (2g+t-2) (\cO_{g-1,t,0} + \cO_{g-1,t-1,0} )
\]
and eq.~(\ref{E:arec}) follows immediately.
\end{proof}

{\bf Remark:} since $\cO_{g,t,0}$-elements can viewed as sets of cycles of
     odd lengths $>1$, we can derive via symbolic methods~\cite{Flajolet:07a}
$$
1+\sum_{g\geq 1} \sum_{t=1}^g\frac{1}{(2g+t)!} a_{g,t} y^{2g+t}x^{t}=
\Big(\frac{1+y}{1-y}\Big)^{\frac{1}{2}x} \exp(-x y).
$$

We proceed by deriving a recurrence for $\kappa_{g,t}$, which is analogous to
the proof for Harer-Zagier recurrence~(\ref{E:bb}) in~\cite{Chapuy:13}.

\begin{theorem}\label{T:kappa}
For any $1\leq t\leq g$, there exists a bijection
\[
n \cR_{g,t,0} \simeq 2 (2n-3)\cdot 2(2n-5) \Big( (n-2)  \cR_{g-1,t,0} +2(2n-7) \cR_{g-1,t-1,0}\Big),
\]
where $n=2g+t$. Therefore $\kappa_{g,t}$ satisfies the recurrence
\begin{equation}\label{E:krec}
(2g+t )\kappa_{g,t}=(2(2g+t)-3) (2(2g+t)-5) \big( (2g+t -2) \kappa_{g-1,t}+2(2(2g+t)-7)\kappa_{g-1,t-1}\big),
\end{equation}
where $\kappa_{1,1}=1$ and $\kappa_{g,t}=0$ if $t<1$ or $t> g$, see Table~\ref{Tab:k}.
\end{theorem}

\begin{table}
\begin{center}
\begin{tabular}{c|ccccc}
\small$\kappa_{g,t}$    & \small$t=1$       & \small$2$         & \small$3$         & \small$4$         & \small $5$ \\
\hline
\small $g=1$            & \small$1$         & \small            & \small            & \small            & \small  \\
\small $2$              & \small$21$        & \small $105$      & \small            & \small            & \small  \\
\small $3$              & \small$1485$      & \small $18018$    & \small$50050$     & \small            & \small  \\
\small $4$              & \small$225225$    & \small $4660227$  & \small$29099070$  & \small $56581525$ & \small  \\
\small $5$              & \small$59520825$  & \small$1804142340$& \small$18472089636$  & \small  $78082504500$ & \small $117123756750$
\end{tabular}
\end{center}
\caption{\small The numbers $\kappa_{g,t}$.}
\label{Tab:k}
\end{table}

\begin{proof}
Let $\cR^\star_{g,t,0}$ denote the set of $\cR_{g,t,0}$ O-trees whith a labeled vertex, $v$.
Let $\mathcal{J}$ and $\mathcal{K}$ denote the subsets $\cR^\star_{g,t,0}$ where the cycle
containing the labeled vertex has length $3$ and length greater than $3$, respectively.

For any $(T,\sigma,v) \in \cR^\star_{g,t,0}$, we have two scenarios
\begin{enumerate}
\item if $(T,\sigma,v)  \in\mathcal{K}$, then the cycle $c$ containing $v$ is of the form
$(v',\ldots,v_1,v_2,v)$. Applying  R\'{e}my's bijection twice to $T$
with respect to $v$ and $v_2$, we obtain the O-tree $(T',\sigma')$ where $T'$ has $n-2$ vertices
and $\sigma'$ is $\sigma$-induced by removing $v$ and $v_2$ from $c$ and subsequent relabeling
$\sigma$ according to $T'$.

The number of possible positions $v_1$ where we can insert $v_2$ and $v$ back is $n-2$, whence
$\mathcal{K} $ is in bijection with $2 (2n-3)\cdot 2(2n-5)\cdot (n-2) \cR_{g-1,t,0}$,

\item if $(T,\sigma,v)  \in\mathcal{J}$, then the cycle $c$ containing $v$ is
of the form $(v_1,v_2,v)$. Applying  R\'{e}my's bijection three times to $T$
with respect to $v$, $v_1$ and $v_2$, we obtain the O-tree $(T',\sigma')$ where
$T'$ has $n-3$ vertices and the O-permutation $\sigma'$ is induced by $\sigma$
by deleting $c=(v_1,v_2,v)$ and relabeling $\sigma$ according to $T'$. Therefore
$\mathcal{J}$ is in bijection with $2 (2n-3)\cdot 2(2n-5)\cdot 2(2n-7) \cR_{g-1,t-1,0}$.
\end{enumerate}

Since $\cR^\star_{g,t,0} \simeq n \cR_{g,t,0}$ and $\cR^\star_{g,t,0}=\mathcal{K}+\mathcal{J} $, we have
a
bijection
\[
n \cR_{g,t,0}\simeq 2 (2n-3)\cdot 2(2n-5)\cdot (n-2) \cR_{g-1,t,0}+ 2 (2n-3)\cdot 2(2n-5)
                                                                     \cdot 2(2n-7) \cR_{g-1,t-1,0},
\]
for any $1\leq t\leq g$.

Since $|\cR_{g,t,0}| =2^{2g} \kappa_{g,t}$, it is clear that this bijection implies eq.~(\ref{E:krec}).
\end{proof}

{\bf Remark:}
$$
1+2\sum_{g\geq 1} \sum_{t=1}^g \frac{\kappa_{g,t}}{(2(2g+t)-3)!!}y^{2g+t}x^{t}=
\Big(\frac{1+y}{1-y}\Big)^{x} \exp(-2x y),
$$
which implies eq.~(\ref{E:krec}).

We next turn to log-concavity of $\{\kappa_{g,t}\}_{t=0}^g$.

\begin{definition}
A sequence $\{a_i\}_{i=0}^n$ of nonnegative real numbers is said to be \emph{unimodal} if
there exists an index $0\leq m \leq n$, called the \emph{mode} of the sequence, such that
$a_0\leq a_1 \leq \cdots  \leq a_{m-1}\leq a_m \geq a_{m+1}\geq \cdots \geq a_n$. The sequence
is said to be \emph{logarithmically concave} (or \emph{log-concave} for short) if
$$
a_i^2\geq a_{i-1}a_{i+1}, \qquad  1\leq i \leq n-1.
$$
\end{definition}
Clearly, log-concavity of a sequence with \emph{positive} terms implies unimodality. Let
us say that the sequence $\{a_i\}_{i=0}^n$ has \emph{no internal zeros} if there do not
exist integers $0\leq i < j < k \leq n$ satisfying $a_i\neq 0 $, $a_j=0$, $a_k\neq 0$.
Then, in fact, a \emph{nonnegative} log-concave sequence with no internal zeros is unimodal.
We call a polynomial $f(x)=\sum_{i=0}^n a_i x^i$ is unimodal and log-concave if the sequence
$\{a_i\}_{i=0}^n$ of its coefficients is unimodal and log-concave, respectively.

\begin{lemma}\label{L:log}
Assume that the number $b_{g,t}$ satisfies the recurrence
$b_{g,t}= p_{g,t} b_{g-1,t}+q_{g,t} b_{g-1,t-1},$
for all $g\geq 1$, where $b_{g,t}$, $p_{g,t}$, $q_{g,t}$ are all nonnegative. If
\begin{itemize}
\item $\{b_{1,t}\}_t$ is log-concave,
\item $\{p_{g,t}\}_t$ and $\{q_{g,t}\}_t$ are log-concave for any $g\geq 1$,
\item $p_{g,t-1} q_{g,t+1}+ p_{g,t+1} q_{g,t-1} \leq 2 p_{g,t} q_{g,t} $ for any $g\geq 1$,
\end{itemize}
then $\{b_{g,t}\}_t$ is log-concave for any $g\geq 1$.
\end{lemma}
\begin{proof}
We prove this by induction on $g$. By induction hypothesis,
$b_{g-1,m} b_{g-1,n} \geq b_{g-1,m-1}b_{g-1,n+1}$ for any $m\leq n$.
For $g$, we expand the product using the recurrence
\begin{eqnarray*}
b_{g,t}^2   &= & (p_{g,t} b_{g-1,t}+q_{g,t} b_{g-1,t-1})^2\\
            &= & p_{g,t}^2 b_{g-1,t}^2 +2 p_{g,t} q_{g,t} b_{g-1,t} b_{g-1,t-1}+ q_{g,t}^2 b_{g-1,t-1}^2
\end{eqnarray*}
and
\begin{eqnarray*}
b_{g,t-1} b_{g,t+1} &= & (p_{g,t-1} b_{g-1,t-1}+q_{g,t-1} b_{g-1,t-2})
                                                 (p_{g,t+1} b_{g-1,t+1}+q_{g,t+1} b_{g-1,t}) \\
   &= & p_{g,t-1} p_{g,t+1} b_{g-1,t-1}  b_{g-1,t+1}+  p_{g,t-1} q_{g,t+1} b_{g-1,t-1}  b_{g-1,t}\\
   &+ &  q_{g,t-1} p_{g,t+1} b_{g-1,t-2} b_{g-1,t+1} + q_{g,t-1}q_{g,t+1} b_{g-1,t-2} b_{g-1,t} .
\end{eqnarray*}
We now compare corresponding terms in the expansion. By assumption and induction hypothesis,
it is clear that
\begin{eqnarray*}
p_{g,t}^2 b_{g-1,t}^2  &\geq & p_{g,t-1} p_{g,t+1} b_{g-1,t-1}  b_{g-1,t+1},\\
q_{g,t}^2 b_{g-1,t-1}^2 &\geq & q_{g,t-1}q_{g,t+1} b_{g-1,t-2} b_{g-1,t}.
\end{eqnarray*}
Also we have
\begin{eqnarray*}
2 p_{g,t} q_{g,t} b_{g-1,t} b_{g-1,t-1}
&\geq&(  p_{g,t-1} q_{g,t+1} +q_{g,t-1} p_{g,t+1}) b_{g-1,t}  b_{g-1,t-1}  \\
&\geq & p_{g,t-1} q_{g,t+1} b_{g-1,t-1}  b_{g-1,t} +q_{g,t-1} p_{g,t+1} b_{g-1,t-2} b_{g-1,t+1},
\end{eqnarray*}
whence the lemma.
\end{proof}

\begin{proposition}
For any fixed $g$, the sequence $\{a_{g,t}\}_{t=0}^g$ is log-concave.
\end{proposition}

\begin{proof}
We just need to verify the conditions in Lemma~\ref{L:log}.
Set $p_{g,t}= q_{g,t}= (2g+t-1) (2g+t-2) $ for $g\geq 1$.
It is clear that $\{p_{g,t}\}_t$ and $\{q_{g,t}\}_t$ are log-concave for any $g\geq
1$. Furthermore, $p_{g,t-1} q_{g,t+1}+ p_{g,t+1} q_{g,t-1} \leq 2 p_{g,t} q_{g,t} $ for
all $g\geq 1$, whence $\{a_{g,t}\}_{t=0}^g$ is log-concave.
\end{proof}
\begin{proposition}\label{P:logk}
For any fixed $g$, the sequence $\{\kappa_{g,t}\}_{t=0}^g$ is log-concave.
\end{proposition}
\begin{proof}
Set $p_{g,t}= \frac{(2(2g+t)-3) (2(2g+t)-5) (2g+t -2)}{2g+t}$ and
$q_{g,t}= \frac{2(2(2g+t)-3) (2(2g+t)-5) (2(2g+t)-7)}{2g+t}$ for $g\geq 1$.
It is clear that $\{p_{g,t}\}_t$ and $\{q_{g,t}\}_t$ are log-concave for any
$g\geq 1$ and $p_{g,t-1} q_{g,t+1}+ p_{g,t+1} q_{g,t-1} \leq 2 p_{g,t} q_{g,t}$
for all $g\geq 1$. Therefore the sequence $\{\kappa_{g,t}\}_{t=0}^g$ is log-concave.
\end{proof}

{\bf Remark:} combining the inductive proof of Lemma~\ref{L:log}
with the bijective proof of recurrences of $a_{g,t}$ and $\kappa_{g,t}$,
we can construct an injection from $\cO_{g,t,0} \times \cO_{g,t,0}$  into
$\cO_{g,t+1,0}   \times \cO_{g,t-1,0}$ and
from $\cR_{g,t,0}   \times \cR_{g,t,0}$  into
$\cR_{g,t+1,0}   \times \cR_{g,t-1,0}$. This provides
combinatorial proofs for the log-concavity of
$\{a_{g,t}\}_{t=0}^g$ and $\{\kappa_{g,t}\}_{t=0}^g$.


\section{Discussion}\label{S:discussion}


Define $\mathcal{L}$ to be an operator acting on the sequence $\{a_i\}_{i=0}^n$
as given by
\[
\mathcal{L}(\{a_i\}_{i=0}^n)=\{b_i\}_{i=0}^n
\]
where $b_i=a_i^2-a_{i-1}a_{i+1} $ for $0\leq i \leq n$ under the convention that
$a_{-1} = a_{n+1} = 0$. Clearly, the sequence $\{a_i\}_{i=0}^n$ is log-concave if and only if the
sequence $\{b_i\}_{i=0}^n$ is nonnegative. Given a sequence $\{a_i\}_{i=0}^n$, we say that it is
$k$-fold log-concave, or \emph{$k$-log-concave}, if $\mathcal{L}^j(\{a_i\}_{i=0}^n)$ is a nonnegative sequence
for any $1\leq j \leq k$. A sequence $\{a_i\}_{i=0}^n$ is said to be \emph{infinitely log-concave} if
it is $k$-log-concave for all $k\geq 1$.

It is well-known that, by Newton's inequality, if the polynomial
 $\sum_{i=0}^n a_i x^i$  with positive coefficients has only real zeros, then
 the sequence $\{a_i\}_{i=0}^n$ is
unimodal and log-concave (see \cite{Hardy:52}).
Such a sequence of positive numbers
whose generating function has only real zeros
is called a \emph{P\'{o}lya frequency} sequence in the theory of total
positivity (see \cite{Karlin:68,Brenti:89,Brenti:95}). Furthermore, we have
\begin{theorem}
If the polynomial $f(x)=\sum_{i=0}^n a_i x^i$
has only real and non-positive zeros,
then the sequence
$\{a_i\}_{i=0}^n$ is infinitely log-concave.
\end{theorem}
This is conjectured independently by Stanley, McNamara--Sagan~\cite{McNamara:10} and Fisk~\cite{Fisk:08},
and proved by Br\"{a}nd\'{e}n~\cite{Branden:11}.

Let $A_g(x)$ denote the generating polynomial of $a_{g,t}$, i.e.,
$A_g(x)= \sum_{t=0}^g a_{g,t} x^{t}$.

\begin{proposition}
For any fixed $g$, polynomial $A_g(x)$ has only real zeros located in $(-1,0]$.
Therefore, the sequence $\{a_{g,t}\}_{t=0}^g$ is infinitely log-concave.
\end{proposition}

\begin{proof}
Set $B_g(x)= x^{2g} A_g(x)= \sum_{t=0}^g a_{g,t} x^{2g+t}$.
It suffices to show that polynomial
$B_g(x)$ has only real zeros located in $(-1,0]$.

Recurrence~(\ref{E:arec}) of $a_{g,t}$  is equivalent to
\[
B_g(x) =  2 x^3 B_{g-1}(x)+2 x^3 (2x+1)\frac{d}{d x} B_{g-1}(x)+ x^4 (x+1)\frac{d^2}{d x^2} B_{g-1}(x),
\]
i.e.,
\begin{equation}\label{E:adif}
B_g(x) =   x^3  \frac{d^2}{d x^2} \Big[x (x+1) B_{g-1}(x)\Big].
\end{equation}
By induction hypothesis, $B_{g-1}(x)$ has all $3g-3$ roots in $(-1,0]$, $2g-1$ of which are at $0$.
Then, applying Rolle's theorem twice we obtain, for the second derivative in eq.~(\ref{E:adif}),
at least $g-2$ roots in $(-1,0)$ and $2g-2$ roots at $0$. Hence,
$B_g(x)$ has all of its $3g$ roots inside $(-1,0]$, $2g+1$ of which are at $0$,
and  $g-1$ roots are in $(-1,0)$.
\end{proof}

Let $K_g(x)$ denote the generating polynomial of $\kappa_{g,t}$, i.e.,
$K_g(x)= \sum_{t=0}^g \kappa_{g,t} x^{t}$.

\begin{conjecture}\label{c:kappa}
For any fixed $g$, polynomial $K_g(x)$ has only real zeros located in $(-\frac{1}{4},0]$.
Therefore, the sequence $\{\kappa_{g,t}\}_{t=0}^g$ is infinitely log-concave.
\end{conjecture}

Similarly, set $H_g(x)= x^{2g} K_g(x)= \sum_{t=0}^g \kappa_{g,t} x^{2g+t}$.
It suffices to show that polynomial
$H_g(x)$ has only real zeros located in $(-\frac{1}{4},0]$.

Recurrence~(\ref{E:krec}) of $\kappa_{g,t}$  is equivalent to
\begin{eqnarray*}
\frac{d}{d x} H_{g}(x) &=&  -6 x^2 H_{g-1}(x)+ 3 x^2 (12 x+1) \frac{d}{d x} H_{g-1}(x)\\
&&+ 12 x^3  (6x+1)\frac{d^2}{d x^2} H_{g-1}(x)+4 x^4 (4 x+1)\frac{d^3}{d x^3} H_{g-1}(x).
\end{eqnarray*}

By Lemma~\ref{L:s}, the generating polynomial of shapes is given by
\begin{eqnarray*}
S_g(x) &=& \sum_{t=1}^{g} \kappa_{g,t} x^{2g+t-1} (1+x)^{2g+t-1}\\
      &=& x^{-1} (1+x)^{-1} \sum_{t=1}^{g} \kappa_{g,t} x^{2g+t} (1+x)^{2g+t}\\
      &=& x^{-1} (1+x)^{-1} H_g(x(x+1)).
\end{eqnarray*}
Therefore Conjecture~\ref{c:kappa} implies that the polynomial $S_g(x)$
has also only real zeros.
\begin{conjecture}\label{T:log}
For any fixed $g$, the generating polynomial $S_g(x)=\sum_{n=2g}^{6g-2}s_g(n) x^n$ of shapes given by
has only real zeros.
Therefore, the sequence $\{s_g(n)\}_{n=2g}^{6g-2}$ is infinitely log-concave.
\end{conjecture}


\section{Appendix}\label{S:appendix}



\subsection{Recursive decomposition of O-trees}


In analogy to the decomposition of C-permutations
and C-decorated trees~\cite{Chapuy:13}, we derive a recursive method to
decompose O-permutations and O-trees. This decomposition can be viewed
also as an analogue to the decomposition of unicellular maps~\cite{Chapuy:11}.

Given an $O$-permutation $\pi$, we can represent $\pi$
as an ordered list of
its cycles, such that all cycles start with its minimal element and are ordered
from left to right such that the minimal elements are in descending order.
We call this representation the \emph{canonical form} of $\pi$.

Let $\mathcal{S}_n$ denote the set of permutations on $[n]$, i.e., sequences of integers.
A \emph{sign sequence} of length $n$ is an $n$-tuple $(i_1,\ldots,i_n)$, where $i_k=\pm$.
\begin{lemma}[Chapuy \textit{et al.}~\cite{Chapuy:13}]\label{lemma:iso_signed}
There is a bijection between permutations on $[n]$
and pairs of an O-permutation on $[n]$ with $n-2g$ cycles
and a sign sequence of length $n-2g-1$, for arbitrary
$0\leq g\leq k=\lfloor \frac{n-1}{2} \rfloor$, i.e.,
\[
\mathcal{S}_{n} \simeq  \biguplus_{g=0}^{k} \{-,+\}^{n-2g-1}\times\cO_{g}(n).
\]
In particular, the O-permutation has one cycle if and only if the
sequence has odd length  and starts with its minimal element.
\end{lemma}
The bijection is illustrated in the following example:
\begin{eqnarray*}
78326154&\rightarrow& 78|3|26|154 \rightarrow 78|3|26|^+(154)\rightarrow 78|36|^-(2)|^+(154)\\
    &\rightarrow& 78|^+(6)|^-(3)|^-(2)|^+(154)  \rightarrow (8)| ^-(7)|^+(6)|^-(3)|^-(2)|^+(154) \\
   &\rightarrow&  (8)(7)(6)(3)(2)(154), (-,+,-,-,+).
\end{eqnarray*}
We adopt the convention that
signed cycles  are represented with the sign preceding the cycle as a exponent, such as $^- (12)$.

\begin{proof}[Proof (Chapuy \textit{et al.}~\cite{Chapuy:13})]
Given a sequence $S \in \mathcal{S}_{n}$, decompose $S=x_1 x_2\ldots x_n$ into blocks
$S_1 S_2\cdots S_l$ as follows: traverse the sequence $S$ from left to right. Set $s_1=x_1$
and $s_i$ to be the first element smaller than all elements traversed before. This procedure
generates blocks $S_i$ that start with $s_i$.

Then we define a process to deal with the blocks successively from right to left.
At each step, we have two cases:
\begin{enumerate}
\item if the block $B$ has odd length, turn $B$ into
the signed cycle $^+(B)$;
\item if $B$ has even length, move the second element $x$ of $B$ out of $B$
and turn $B$ into the signed cycle $^-(B)$. If $x$ is the minimum of the
elements to the left of $B$, set $\{x\}$ to be a singleton-block before $^-(B)$
and append $x$ at the end of the block preceding $B$, otherwise.
\end{enumerate}
This right-to-left process ends up with the last block $B$ having odd length
and we produce $(B)$ as the last cycle. This process generates a sign sequence,
$I$, together with an O-permutation, $\pi$.

By construction, $\pi$ is represented in its canonical form and furthermore the number of
signs generated by the process is one less than the number of cycles of the O-permutation.
Accordingly the process defines the mapping
$$
\Phi \colon
\mathcal{S}_{n} \to \biguplus_{g=0}^{k} \{-,+\}^{n-2g-1}\times\cO_{g}(n), \qquad
S\mapsto (I,\pi).
$$
Conversely, given an $O$-permutation $\pi$ with $n-2g$ cycles and a sign sequence $I$
of length $n-2g-1$, write $\pi$ in its canonical form.
Assign each cycle except of the leftmost one with the corresponding sign from the
sign sequence $I$.
Turn the leftmost unsigned cycle $(B)$ into the block $B$.
Then treat the signed cycles from left to right, starting with the second one, as follows:
let $^{\epsilon}(B)$ be the signed cycle to be processed and let $B'$ be the block to the left
of $^{\epsilon}(B)$. Process $^{\epsilon}(B)$ into the block $B$, by either just removing the sign
if $\epsilon=+$ or by removing the sign $\epsilon=-$ and moving the last element of $B'$ to the
second position of $B$. This generates an ordered list of blocks, which can be viewed as a
sequence $S$, i.e.~we have
$$
\Psi \colon \biguplus_{g=0}^{k} \{-,+\}^{n-2g-1}\times\cO_{g}(n)  \to  \mathcal{S}_{n} , \qquad
(I,\pi) \mapsto S.
$$
By construction, $\Psi \circ \Phi=\text{\rm id}$ and $\Phi \circ \Psi=\text{\rm id}$.
\end{proof}

An element of an O-permutation is called \emph{non-minimal} if it is
not the minimum in its cycle. Non-minimal elements play the same role for
O-permutations (and O-trees) as trisections for unicellular maps~\cite{Chapuy:11}.
Indeed, an O-permutation of genus $g$ has $2g$ non-minimal elements
(Lemma 3 in~\cite{Chapuy:11}), and moreover we have
Proposition~\ref{prop:2} and Proposition~\ref{prop:3},
which are an analogue of Proposition~\ref{prop:1}.

\begin{proof}
[Proof of Proposition~\ref{prop:2} (Chapuy \textit{et al.}~\cite{Chapuy:13})]
For $k\geq 1$, let $\cO_g^\star(n)$ be  the set of O-permutations from $\cO_g(n)$
having one labeled non-minimal element. Note that
$\cO_g^\star(n) \simeq 2g\ \!\cO_g(n)$ since an O-permutation
in $\cO_g(n)$ has $2g$ non-minimal elements.

Given $\pi\in\cO_g^\star(n)$, we write the cycle containing the labeled element
$i$ of $\pi$ as a sequence beginning with $i$ and apply bijection $\Phi$ in
Lemma~\ref{lemma:iso_signed}.
This gives a collection $S'$ of $(2k+1)\geq 3$ cycles of odd length,
together with a sign-sequence $I$ of length $2k$.
Hence, replacing the cycle containing the labeled element $i$ with these
$(2k+1)$ cycles, we obtain an O-permutation $\pi'$ of genus $g-k$ with
$2k+1$ labeled cycles.

We have thus shown that $\cO_g^\star(n) \simeq\biguplus_{k=1}^g \{-,+\}^{2k}
\times \cO_{g-k}^{(2k+1)}(n)
\simeq  \biguplus_{k=1}^g  2^{2k} \cO_{g-k}^{(2k+1)}(n)$.
By construction of $\Phi$, the cycles of $\pi$ are obtained from the cycles of $\pi'$
by merging labeled cycles in $S'$ into a single cycle  and the proposition follows.
\end{proof}

\begin{proof}[Proof of Proposition~\ref{prop:3} (Chapuy \textit{et al.}~\cite{Chapuy:13})]
We have by definition $\cT_g(n)=\cE_0(n)\times\cO_g(n+1)$ and
Proposition~\ref{prop:2} guarantees
$2g\ \!\cO_g(n) \simeq\biguplus_{k=1}^g \{-,+\}^{2k}\times \cO_{g-k}^{(2k+1)}(n)$.
Therefore we have
$$
2g\ \!\cT_g(n)\simeq\biguplus_{k=1}^g \{-,+\}^{2k}\times \cT_{g-k}^{(2k+1)}(n)
\simeq\biguplus_{k=1}^g 2^{2k}\cT_{g-k}^{(2k+1)}(n).
$$
The statement about the underlying graphs follows from the fact that
the bijection $\Phi$ in Lemma~\ref{lemma:iso_signed} merges
the labeled cycles into a unique cycle.
\end{proof}

\begin{proof}[Proof of Theorem~\ref{T:main} (Chapuy \textit{et al.}~\cite{Chapuy:13})]
We fix $n$ and prove the theorem by induction on $g$ .
The case $g=0$ is obvious, as there is only one $O$-permutation of size
$(n+1)$ and genus $0$, i.e., the identity permutation and both sides are the
set of plane trees with $n$ edges.

Assume $g>0$. The induction hypothesis ensures that for each $g'<g$,
$2^{2g'}\cE_{g'}^{(2k+1)}(n)\simeq \cT_{g'}^{(2k+1)}(n)$, where the underlying graphs
of the corresponding objects are by construction the same. Thus we have
$$
\biguplus_{k=1}^g 2^{2k}\cdot 2^{2(g-k)} \cE_{g-k}^{(2k+1)}(n)\simeq
\biguplus_{k=1}^g 2^{2k} \cdot \cT_{g-k}^{(2k+1)}(n).
$$
Combining this with eq.~(\ref{eq:bi1}) of Proposition~\ref{prop:1}
and eq.~(\ref{eq:bi2}) of Proposition~\ref{prop:3}, we derive
$$
2g\ \!2^{2g}\cE_g(n)\simeq 2g\ \!\cT_g(n),
$$
where the underlying graphs of corresponding objects are the same.
Note that by construction of corresponding bijections
in Propositions~\ref{prop:1} and~\ref{prop:3},
the $2g$ factor never affect the underlying graphs of corresponding
objects.
Hence, we can extract from this $2g$-to-$2g$ correspondence a
$1$-to-$1$ correspondence, i.e., $ 2^{2g}\cE_g(n)\simeq \cT_g(n)$,
which also preserves the underlying graphs of corresponding
objects. The following diagram
$$
\diagram
2^{2g'} \cE_{g'}(n) \dline \rto  & \cT_{g'}(n) \\
2^{2g'}\cE_{g'}^{(2k+1)}(n)  \dline \rto & \cT_{g'}^{(2k+1)}(n)\uline\\
\biguplus_{k=1}^g 2^{2k}\cdot 2^{2(g-k)} \cE_{g-k}^{(2k+1)}(n) \dline \rto &  \biguplus_{k=1}^g 2^{2k} \cdot \cT_{g-k}^{(2k+1)}(n)\uline \dline\\
2g\ \!2^{2g}\cE_g(n)\dline \rto & 2g\ \!\cT_g(n) \uline\\
2^{2g}\cE_g(n)  \rto & \cT_g(n)\uline
\enddiagram
$$
depicts the construction of the bijection.
\end{proof}

\section*{Acknowledgments}
We acknowledge the financial support of the Future and Emerging Technologies (FET) programme
within the Seventh Framework Programme (FP7) for Research of the European Commission, under
the FET-Proactive grant agreement TOPDRIM, number FP7-ICT-318121.

\bibliographystyle{plain}  
\bibliography{Sp}     

\end{document}